\documentclass[10pt, oneside]{amsart}

\usepackage{amsmath,amsthm,amsfonts}
\usepackage{mathtools}
\usepackage{amssymb}
\usepackage{mathrsfs}
\usepackage{enumitem}
\usepackage{tikz-cd}
\usepackage{dsfont}
\usepackage[english]{babel} 
\usepackage{fullpage}
\usepackage{afterpage}
\usepackage{hyperref}
\usepackage[T1]{fontenc}
\usepackage[utf8]{inputenc}
\usepackage{lmodern}
\usepackage{graphicx} 
\usepackage{multicol} 
\usepackage{tikz} 
\usepackage{caption}
\usepackage{pgfplots} 
\usepackage{wrapfig}
\usepackage{array}
\usepackage{stackrel}
\usepackage{extarrows}
\usepackage{pgfplots}
\usetikzlibrary{decorations.markings}
\usepackage{mathtools}

\newcommand{\Z}{\mathbb{Z}} 
\newcommand{\N}{\mathbb{N}} 

\newcommand{\im}{\operatorname{Im}}
\newcommand{\Ext}{\operatorname{Ext}}

\newcommand{\F}{\mathbb{F}}
\newcommand{\Q}{\mathbb{Q}}

\newcommand{\gr}{\operatorname{gr}} 

\newcommand{\lk}{\operatorname{lk}}

\newtheorem{theorem}{Theorem}[section]
\newtheorem{lemma}[theorem]{Lemma}
\newtheorem{corollary}[theorem]{Corollary}
\newtheorem{proposition}[theorem]{Proposition}
\theoremstyle{definition}
\newtheorem{definition}[theorem]{Definition}
\theoremstyle{remark}
\newtheorem{remark}[theorem]{Remark}

\theoremstyle{plain}
\newtheorem{teoInt}{Theorem}

\pgfplotsset{compat=1.18}

\title[From RAAGs to Artin groups]{Cohomology rings and $p$-local behavior of even Artin groups.}
\author{Marcos Escartín Ferrer}
\author{Giorgio Leoni}
\author{Conchita Martínez Pérez}
\date{today}



\thanks{\noindent
The first and third authors are partially supported by the Spanish Government PID2021-126254NB-I00 and PID2024-155800NB-C32 and also by Departamento de Ciencia, Universidad y Sociedad del 
Conocimiento del Gobierno de Arag{\'o}n (grant code: E22-23R: ``{\'A}lgebra y Geometr{\'i}a'')}

\begin{document}

\begin{abstract} We generalize to certain families of even Artin groups several classical results on right-angled Artin groups. In particular, we compute the cohomology ring, describe the pro-$p$ completion, and determine the $p$-Zassenhaus restricted Lie algebra in the FC case. As a by-product, we prove a rigidity result that implies that if two even Artin groups of FC type are isomorphic, then for every prime $p$, the $p$-parts of their defining graphs are isomorphic.

\end{abstract}

\maketitle

\section{Introduction}\label{sec:introduction}

Right-angled Artin groups (RAAGs) form a well-known family of groups that can be defined in terms of a finite simplicial graph $\Gamma$ and interpolate between free groups and free abelian groups. In recent years, considerable effort has been devoted to understanding which properties of RAAGs extend to the much more mysterious class of Artin groups, also called sometimes Artin-Tits groups. These efforts have been partially successful, in the sense that one can often identify particular families of Artin groups, according to the property under consideration, to which certain results for RAAGs extend.

In this paper, we continue this line of research. We focus on cohomological properties, which are very well understood in the case of RAAGs. Bartholdi, Härer, and Schick give in \cite{BHS} a very nice and self-contained account of many of the properties of RAAGs that we consider here.

For example, if \(A_\Gamma\) is the RAAG associated with the defining graph \(\Gamma\) and \(F\) is a field, then the cohomology ring \(H^\bullet(A_\Gamma,F)\) admits an explicit description in terms of \(\Gamma\): it is the exterior \(F\)-algebra generated by the dual classes \(\gamma_a\) of the standard generators of \(A_\Gamma\) (that is, of the vertices of \(\Gamma\)), subject to the relations \(\gamma_a\gamma_b=0\) whenever \(\{a,b\}\) is a non-edge of \(\Gamma\). In particular, \(H^\bullet(A_\Gamma,F)\) is quadratic and, moreover, it is {\sl Koszul}. A quadratic graded \(F\)-algebra \(S\) is said to be Koszul if the Yoneda algebra \(\mathrm{Ext}_S(F,F)\) and \(S\) itself are {\sl quadratic duals} of each other.

In the case of \(H^\bullet(A_\Gamma,F)\), its quadratic dual is also isomorphic to another important \(F\)-algebra associated with \(A_\Gamma\): the universal enveloping algebra \(U_\Gamma\) of the Lie algebra \(\gr_\bullet(A_\Gamma)\otimes_\mathbb{Z} F\), where \(\gr_\bullet(A_\Gamma)\) denotes the {\sl descending central series Lie ring} of \(A_\Gamma\) (see Section~\ref{sec:central} for definitions).

This fact is closely related to the existence of a {\sl Magnus map}. The terminology comes from the classical result of Magnus stating that, if \(G\) is a free group, then there exists a faithful representation of \(G\) into the group of units of the ring of formal power series with integer coefficients. This was later extended to RAAGs by Droms in his PhD thesis \cite{Droms}; see also \cite{Wade}. Using further work of Duchamp and Krob, one deduces that the Lie ring \(\gr_\bullet(A_\Gamma)\) is torsion-free and admits the same presentation as the group \(A_\Gamma\), but in the category of Lie rings.

A natural question is to what extent this favorable behavior of RAAGs persists for arbitrary Artin groups. A first difficulty arises from the following fact: for RAAGs, there exists a particularly nice model for the classifying space, namely the Salvetti complex, which can be used to describe the cohomology ring. For arbitrary Artin groups, there is also a Salvetti complex, but in general it remains an open problem whether this complex is indeed a classifying space. This assertion is equivalent to the so-called {\sl \(K(\pi,1)\)-conjecture}, which is a central problem in the area.

We will use the following notation for Artin groups. As before, \(\Gamma\) will denote a finite simplicial graph, but now we assume that the edges of \(\Gamma\) are labeled by integers greater than or equal to \(2\), encoding the Artin presentation (see Section~\ref{sec:Artin} for details). The associated group will be denoted by \(A_\Gamma\).

Even if a given Artin group is assumed to satisfy the \(K(\pi,1)\)-conjecture, the computation of its cohomology ring can be rather complicated. To illustrate this, see \cite{Landi 2000}, where Landi computes explicitly the cohomology rings of certain spherical Artin groups using a chain complex constructed by Salvetti \cite{Salvetti 1994}, whose cohomology computes the cohomology of any Artin group satisfying the \(K(\pi,1)\)-conjecture.

However, we show that, if one restricts to {\sl even} Artin groups, that is, Artin groups \(A_\Gamma\) such that all labels in \(\Gamma\) are even integers, then the cohomology ring admits a much simpler description and, in fact, resembles the cohomology ring of a RAAG. More precisely, we prove the following.

\begin{teoInt} \label{teo:CohomRing}
{\sl Let \(A_\Gamma\) be an even Artin group and \(R\) a unital commutative ring. Fix an order \(<\) on the set of vertices \(V(\Gamma)\). Then \(H^\bullet(A_\Gamma,R)\) is the exterior graded ring generated by the degree \(1\) elements \(\sigma_v\) for \(v\in V(\Gamma)\) and the degree \(2\) elements \(\sigma_e\) for edges \(e=\{u,v\}\in E(\Gamma)\) with label greater than \(2\), subject to the following relations:
\begin{itemize}
\item[(1)] \(\sigma_u\sigma_v=m\sigma_e\) if \(e=\{u,v\}\) has label \(2m\) and \(u<v\),
\item[(2)] \(\sigma_{X_1}\sigma_{X_2}=0\) if \(X_1\cap X_2\neq\emptyset\) or \(X_1\cup X_2\) is not spherical.
\end{itemize}}
\end{teoInt}

A key property shared by even Artin groups and RAAGs, and which plays a central role in the proof, is that for every subgraph \(\Delta\subseteq \Gamma\), the group \(A_\Delta\) is a retract of \(A_\Gamma\). In Section~\ref{sec:Isomorphism}, we give an example of two non-isomorphic even Artin groups having isomorphic cohomology rings; to prove that these groups are non-isomorphic, we use \(\Sigma\)-invariants (also called BNRS-invariants).

A first consequence of Theorem~\ref{teo:CohomRing} is that, for even Artin groups, the cohomology ring may or may not be generated by degree one elements, depending on the coefficient ring. Another difference from the case of RAAGs appears when one considers descending central series Lie rings: for arbitrary Artin groups these are no longer torsion-free and seem difficult to describe.

Things simplify when working with coefficients in a field of characteristic zero, say \(\mathbb{Q}\), since in that case even Artin groups behave very much like RAAGs. However, this is not very interesting, because both the cohomology ring \(H^\bullet(A_\Gamma,\mathbb{Q})\) and the rational descending central series Lie algebra \(\gr_\bullet(A_\Gamma)\otimes_\mathbb{Z}\mathbb{Q}\) coincide with the corresponding objects for a RAAG that can be naturally associated to our Artin group: namely, the quotient of \(A_\Gamma\) obtained from the same defining graph by replacing all labels by \(2\) (see Section~\ref{sec:Artin} for details).

For this reason, throughout most of the paper we work with coefficients in a field \(F\) of prime characteristic \(p\) and study the groups from a \(p\)-local point of view. From this perspective, many of the properties we consider for a group \(G\) can be interpreted as properties of its pro-\(p\) completion
\[
\widehat{G}=\varprojlim_{U\leq_p G} G/U,
\]
where \(U\leq_p G\) means that \(U\) is a subgroup of finite \(p\)-power index in \(G\). In the case of Artin groups, we show that the pro-\(p\) completion admits the following description.


\begin{teoInt}\label{teo:prop}
{\sl Let $A_\Gamma$ be an Artin group and $p$ a prime. Then the pro-$p$ completion of $A_\Gamma$ is the group given by the same presentation as $A_{\Gamma_p}$, but in the category of pro-$p$ groups.}
\end{teoInt}

The graph $\Gamma_p$ is defined in Section~\ref{sec:central}; it is essentially the {\sl even $p$-part} of $\Gamma$, so the Artin group $A_{\Gamma_p}$ is always even.

Typically, there are two $p$-central series that one can associate to a finitely generated group $G$: one was first defined by Stallings and is sometimes called the Stallings $p$-central series, while the other, which is the central series we consider here, was first defined by Zassenhaus and is called the $p$-Zassenhaus central series, denoted $\gamma_i^{[p]}(G)$. One can associate Lie algebras to these central series; in the case of the $p$-Zassenhaus central series we obtain a $p$-restricted Lie algebra, denoted $\mathrm{gr}_\bullet^{[p]}(G)$.

A discrete group and its pro-$p$ completion have the same $p$-Zassenhaus restricted Lie algebra. In this context, we say that an Artin group $A_\Gamma$ is pro-$p$ FC if its pro-$p$ completion is of type FC, meaning that the group $A_{\Gamma_p}$ is FC. For these groups we give the following presentation of the $p$-Zassenhaus restricted Lie algebra.

\begin{teoInt}[Restricted $p$-algebras of Artin groups]
\label{teo:prestricted}
{\sl Let $p$ be a prime and $A_\Gamma$ a pro-$p$ FC Artin group. Then the $p$-Zassenhaus Lie algebra of $A_\Gamma$ is the $p$-restricted Lie algebra generated by $V(\Gamma)$ subject to the following relations:
\begin{itemize}
\item[1)] $\mathfrak{a} = \mathfrak{b}$ if $\{\mathfrak{a},\mathfrak{b}\}$ forms an edge of $\Gamma$ with odd label,

\item[2)] $[(\mathfrak{a}+\mathfrak{b})^{[p]^t}, \mathfrak{b}] = 0$ if $\{\mathfrak{a},\mathfrak{b}\}$ forms an edge of $\Gamma$ with even label $2kp^t$, where $p \nmid k$.
\end{itemize}}
\end{teoInt}

From Theorems $\ref{teo:CohomRing}$ and $\ref{teo:prestricted}$ we see that we cannot expect any duality between the cohomology ring in characteristic $p$ of an Artin group and its $p$-Zassenhaus restricted Lie algebra. However, at least in the pro-$p$ FC case, we can extract a lot of cohomological information from these groups. To do so we introduce the following notion.

Let $U_G$ be the universal enveloping algebra of $\mathrm{gr}^{[p]}_\bullet(G)$, $\hat{U}_G$ the degree completion of $U_G$, and $\mathbb{F}_p[[G]]$ the completion of the group ring $\mathbb{F}_p G$ with respect to the powers of the augmentation ideal, where $\mathbb{F}_p$ is the field of $p$ elements. We say that a group $G$ is $p$-Magnus if there is a faithful representation
$G \to \hat{U}_G$
that extends to an isomorphism of filtered algebras
$\mathbb{F}_p[[G]] \to \hat{U}_G$.

If a group $G$ is $p$-Magnus, there is a May-type spectral sequence that collapses at the first page and allows one to compute the continuous cohomology groups $dH^\bullet(G,\mathbb{F}_p)$ in terms of the bigraded Ext functors $\mathrm{Ext}_{U_G}(\mathbb{F}_p,\mathbb{F}_p)$. We show in Theorem $\ref{teo:Giorgiofiltered}$ that being $p$-Magnus is preserved under free amalgamated products along retracts, and as a consequence we obtain:

\begin{teoInt}
\label{teo:filtered}
{\sl Let $p$ be a prime and $A_\Gamma$ a pro-$p$ FC Artin group. Then $A_{\Gamma_p}$ is residually-$p$, $p$-Magnus, and cohomologically $p$-complete.}
\end{teoInt}

The notions of being residually-$p$ and cohomologically $p$-complete are standard and are explained in Sections~\ref{sec:res} and \ref{sec:complete}. This result generalizes the fact that RAAGs are cohomologically $p$-complete for any $p$ \cite[Theorem 2.6]{L}.

The fact that, for a pro-$p$ FC Artin group $A_\Gamma$, the group $A_{\Gamma_p}$ is residually-$p$ follows from  \cite{Moldavanskii} and \cite{BS}. In fact, these results imply that, for pro-$p$ FC Artin groups, $A_{\Gamma_p}$ is the largest quotient of $A_\Gamma$ that is residually-$p$. As a consequence, and using also the main result in \cite{BlascoParis}, we obtain the following theorem related to the rigidity problem for Artin groups.

\begin{teoInt}\label{teo:isoproblem}
{\sl Let $A_{\Gamma}$ and $A_\Delta$ be pro-$p$ FC Artin groups such that $A_{\Gamma}$ and $A_\Delta$ are isomorphic. Then there is a graph isomorphism $\Gamma_p \cong \Delta_p$.}
\end{teoInt}

As one might expect, this theorem cannot be reversed. In fact, even if $\Gamma_p \cong \Delta_p$ for every prime $p$, the groups $A_\Gamma$ and $A_\Delta$ may still fail to be isomorphic. To make this explicit, we use again $\Sigma$-invariants to give an example of two non-isomorphic groups which are pro-$p$ FC Artin groups and have isomorphic $p$-parts of the defining graph for all $p$.


In the last section, Section~\ref{sec:Hydra}, we give an example illustrating how different the residual properties of groups with the same $p$-Zassenhaus restricted Lie algebra can be. More precisely, in Theorem~\ref{teo:Hydra} we construct, for each pro-$p$ FC Artin group $A_\Gamma$, a residually torsion-free nilpotent group $G$ such that $G$ and $A_\Gamma$ have the same $p$-Zassenhaus restricted Lie algebra. Note that Artin groups are residually torsion-free nilpotent if and only if they are RAAGs (see Section~\ref{sec:Artin}).

The structure of the paper is as follows. In Section~2 we review the necessary background on even Artin groups and prove Theorem~\ref{teo:CohomRing}. Section~3 is devoted to the study of $p$-Zassenhaus restricted Lie algebras of Artin groups, where we prove Theorems~\ref{teo:prop} and~\ref{teo:prestricted}. In Section~4 we establish Theorem~\ref{teo:isoproblem} and prove the residually-$p$ part of Theorem~\ref{teo:filtered}. Section~5 introduces the notion of $p$-Magnus groups together with the necessary preliminary results and we prove the $p$-magnus part of \ref{teo:filtered}. In Section~6 we complete the proof of Theorem~\ref{teo:filtered} by establishing the cohomological $p$-completeness property. Finally, the last section provides the example mentioned above of a residually torsion free nilpotent groups with the same $p$-Zassenhaus restricted Lie algebras as a prescribed  pro-$p$ FC Artin group.

\section{Even Artin groups}\label{sec:Artin}

There are several ways to define Artin groups, for example one may use Dynkin diagrams. Another option is to use a Coxeter matrix that determines the defining relators, or by analogy with RAAGs, the same info can be encoded in a labeled graph $\Gamma$ where the labels correspond to the finite entries of the Coxeter matrix. Here and throughout the paper, graph means finite simplicial labeled graph, where the labels are in $\Z_{\geq 2}$. We denote by $V(\Gamma)$ the set of vertices and by $E(\Gamma)$ the set of edges of $\Gamma$. 
With this point of view, we define the {\sl Artin group} $A_\Gamma$ as the group given by the presentation
$$A_\Gamma=\langle V(\Gamma)\mid ab\buildrel{m}\over\ldots=ba\buildrel{m}\over\ldots\text{ for each }\{a,b\}\in E(\Gamma)\text{ with label }m\rangle.$$
By a subgraph $\Delta\subseteq\Gamma$ we always mean an {\sl induced subgraph}, meaning that if two vertices of $\Delta$ are linked in $\Gamma$, then they are also linked in $\Delta$ (and with the same label). Often, will identify $\Delta$ with its set of vertices. 
The {\sl Coxeter group} associated to $\Gamma$ is the group
$$W_\Gamma=\langle A_\Gamma\mid a^2=1\text{ for each }a\in V(\Gamma)\rangle$$

The Artin group $A_\Gamma$ (or the graph $\Gamma$) is called {\sl spherical} if $W_\Gamma$ is finite. If the defining graph is a single edge the Artin group is called {\sl dihedral}, dihedral Artin groups are spherical.
An Artin group $A_\Gamma$ is of {\sl type FC} if for every complete subgraph $\Delta\subseteq\Gamma$ the group $A_\Delta$ is spherical, {\sl right angled Artin group} if all the labels are equal to two and {\sl even} if all labels are even numbers. In this last case we will also say that the graph $\Gamma$ is even. 

An important fact that we will often use is that the subgroup of $A_\Gamma$ generated by $\Delta$ is isomorphic to the group $A_\Delta$. If $A_\Gamma$ is even, the subgroup $A_\Delta$ is moreover a retract: to see it observe that the map 
$$\begin{aligned}
    A_\Gamma&\to A_\Delta\\
    a&\mapsto\Big\{\begin{aligned}
        &a\text{ if }a\in\Delta\\
        &1\text{ otherwise}.\\
    \end{aligned}\\
\end{aligned}$$
is well defined and restricts to the identity on $A_\Delta$.

There is another property of even Artin groups that will be crucial in this paper. A spherical graph $\Gamma$ is irreducible when it can not be decomposed as a union of two proper disjoint subgraphs $\Gamma_1$ and $\Gamma_2$ such that the edges conecting vertices in $\Gamma_1$ to vertices in $\Gamma_2$ have all label 2 (observe that if $\Gamma$ is spherical, it must be complete). There is a well known classification of irreducible spherical graphs in terms of irreducible Dynkin diagrams (cf. \cite{Coxeter 1935}). It is a consequence of this classification that there are only two types of spherical even graphs: either single vertices or single edges. In the case of a single edge $e$ with label, say, $2m$, the corresponding dihedral Artin group is
$$A_e=\langle a,b\mid (ab)^m=(ba)^m\rangle.$$
This has the following consequence: if $A_\Gamma$ is even and $\Delta\subseteq\Gamma$ is spherical, then $A_\Delta$ is a direct product of groups which are either infinite cyclic or dihedral even Artin groups.

\subsection{Cohomology groups for even Artin groups}

Although families such as right-angled Artin groups or spherical Artin groups are relatively well-known, there are very few results known for arbitrary Artin groups. A very important open problem in the area is the $K(\pi,1)$-conjecture, if true, it would imply that Artin groups have many properties which currently are conjectural, for example it would imply that all the Artin groups are torsion-free, a fact that somehow surprisingly is not known in general. As stated in the introduction, for each Artin group there is a complex called the Salvetti complex such that the $K(\pi,1)$-conjecture is equivalent to the statement that the Salvetti complex is a classifying space for the group. This complex can be described as the CW-complex constructed from the presentation complex of $A_\Gamma$ by attaching cells corresponding to the spherical subgraphs of $\Gamma$ but we will not need the explicit construction of this complex here. Instead, we will use a chain complex constructed by Salvetti whose cohomology groups are the cohomology groups $H^i(A_\Gamma,R)$ where $R$ is a commutative unital ring. To describe this complex we will use the following technical result.

\begin{proposition}[\cite{Bourbaki 2008}, Chapter IV, Exercise 3]
	Let $W_\Gamma$ be a Coxeter group and $\Delta\subset\Gamma$ a subgraph. Then, for every coset in $W_\Gamma/W_{\Delta}$, there exists a unique representative $w\in W_\Gamma$ of minimal word length in its class $wW_{\Delta}$.
\end{proposition}

\noindent
The elements of the (left) coset space $W_\Gamma/W_{\Delta}$ that have minimal word length in their respective classes are called \textbf{$\Delta$-reduced elements}, and the set of all such elements is denoted by $W_\Gamma^{\Delta}$.

\begin{theorem}[\cite{Salvetti 1994}, Theorem 1.8]\label{Cochaincomplex}
	Let $A_\Gamma$ be an Artin group satisfying the $K(\pi,1)$-conjecture.  
	Define a cochain complex $(C^\ast,\partial^\ast)$ by
	$$C^k = \bigoplus_{\substack{X\subset\Gamma \ \text{spherical}\\ |V(X)|=k}} R\,\sigma_X,$$
	with differential
	$$\partial(\sigma_X) \;=\; 
	\sum_{\substack{v\in V(\Gamma\setminus X) \\ X\cup\{v\}\ \text{spherical}}}
	\langle X \mid X\cup\{v\}\rangle
	\sum_{w\in W^{X}_{X\cup\{v\}}} (-1)^{l(w)}\, \sigma_{X\cup\{v\}}.$$
	Then, $H^\ast(A_\Gamma;\,R) \;\cong\; H^\ast(C^\ast)$.
\end{theorem}

Next, we show that the cochain complex of Theorem \ref{Cochaincomplex} simplifies considerably in the case of even Artin groups. 
Moreover, even Artin groups are known to satisfy the $K(\pi,1)$-conjecture (cf. \cite{Charney 2004}), so we get:

\begin{lemma}\label{CohomologyEven}
	If $A_\Gamma$ is an even Artin group, then all differentials in the cochain complex of Theorem \ref{Cochaincomplex} vanish, and hence
    $$H^n(A_\Gamma;R) \;=\; \bigoplus_{\substack{\Delta\subset\Gamma \ \text{spherical}\\ |V(\Delta)|=n}} R\,\sigma_\Delta.$$
\end{lemma}

\begin{proof}
	Since $\Gamma$ is even, for each spherical $X\subset\Gamma$, the group $A_X$ is a direct product of groups which are either infinite cyclic or even dihedral. Moreover, if $A_X=A_{X_1}\times A_{X_2}$ and $v\in V(X_1)$, it follows that
	$$W^{(X_1\cup X_2)\setminus\{v\}}_{X_1\cup X_2}
	= W^{X_1\setminus\{v\}}_{X_1}.$$
		In particular, it suffices to check that $\sum\limits_{w\in W^{\emptyset}_{v}} (-1)^{l(w)}=0$ for each $v\in V(\Gamma)$ and that $\sum\limits_{w\in W^{u}_{\lbrace u,v\rbrace}} (-1)^{l(w)}=0$ for each $\lbrace u,v\rbrace\in E(\Gamma)$. Since $W^\emptyset_v=\{1,v\}$ for all $v\in V(\Gamma)$, we obtain
	
	$$\sum\limits_{w\in W^{\emptyset}_{v}} (-1)^{l(w)}=(-1)^{l(1)}+(-1)^{l(v)}=0$$
	Now, let $v\in V(\Gamma)$ and $e=\{v,w\}\in E(\Gamma)$ an edge of label, say, $2m$. Then, $W_e$ has order $4m$ so $W_e^v=\{1,w,vw,\dots,wv\cdots w\}$ has an even number of elements thus
	$$\sum\limits_{w\in W^{v}_{\lbrace v,w\rbrace}} (-1)^{l(w)}=(-1)^{l(1)}+(-1)^{l(w)}+\cdots+(-1)^{l(wv\cdots wv)}=1-1+1-1+\cdots+1-1=0.$$
	 This implies the result.
\end{proof}

\begin{remark} A similar argument shows $H_n(A_\Gamma;R)=\bigoplus_{\substack{\Delta\subset\Gamma \ \text{spherical}\\ |V(\Delta)|=n}} R\,\sigma_\Delta$.
\end{remark}

\subsection{Cohomology ring of even Artin groups}

Now we want to determine the ring structure of $H^\bullet(A_\Gamma,R)$ for $A_\Gamma$ an even Artin group. In the case when $A_\Gamma$ is a RAAG, $H^\bullet(A_\Gamma,R)$ is the exterior ring generated by the 1-dimensional classes $\sigma_v$ for $v\in V(\Gamma)$ with relations $\sigma_v\sigma_w=0$ whenever $v,w$ do not form an edge of $\Gamma$ (cf. \cite[Proposition 2.4]{Koberda 2022}). 
The ring structure of $H^\bullet(A_\Gamma,R)$ is also known for spherical Artin groups (cf. \cite{Landi 2000}). In particular, when $A_e$ is the even dihedral group associated to the edge $e=\{u,v\}$ with label $2m$, from results in \cite{Landi 2000} one deduces that $H^\bullet(A_e,R)$ is the exterior ring generated by the 1-dimensional classes $\sigma_u,\sigma_v$ and the 2-dimensional class $\sigma_e$ with relations $\sigma_u\sigma_v=m\sigma_e$, $\sigma_u\sigma_e=\sigma_v\sigma_e=0$.

\begin{lemma}\label{CupprodLemma}
	Let $A_\Gamma$ be an even Artin group and fix a total order in $V(\Gamma)$. With the notation of Lemma \ref{CohomologyEven} we have
	\begin{enumerate}
		\item[i)] If $e=\{u,v\}\in E(\Gamma)$ has label $2m$ and $u<v$, then $\sigma_u\sigma_v=m\sigma_{e}$.
		\item[ii)] If $\emptyset\neq X_1,X_2\subset\Gamma$ are spherical such that either $X_1\cup X_2$ is non-spherical or $X_1\cap X_2\neq\emptyset$, then $\sigma_{X_1}\sigma_{X_2}=0$.
		\item[iii)] If $X_1,X_2\subset\Gamma$ are spherical such that $X_1\cap X_2=\emptyset$ and $A_{X_1}$, $A_{X_2}$ commute, then $X_1\cup X_2$ is spherical and $\sigma_{X_1}\sigma_{X_2}=\pm\sigma_{X_1\cup X_2}$.
	\end{enumerate}
\end{lemma}
\begin{proof}
	Since $A_\Gamma$ is even, every special subgroup is a retract. In particular, for each subgraph $X\subset\Gamma$, the retraction $A_\Gamma\to A_X$ induces an injective map
	$$H^\ast(A_X,R)\to H^\ast(A_\Gamma,R).$$
	
	For i), take $X=\{u,v\}$. Then $A_X$ is a dihedral Artin group so the result follows from \cite{Landi 2000} (see above).
	
	For ii), assume first that $X=X_1\cup X_2$ is not spherical, then Theorem~\ref{Cochaincomplex} implies that for $m=|X|$, $H^{m}(A_X,R)=0$, so the cup product must be zero. Similarly, if $X_1\cap X_2\neq\emptyset$, then $|X_1|+|X_2|>m=|X|$ so again the cup product must be zero in $H^\bullet(A_X,R)$.
	
	For iii), observe that $X=X_1\cup X_2$ is always spherical and the cup product behaves well with respect to direct products by the Künneth formula.
\end{proof}

We are ready to prove Theorem \ref{teo:CohomRing}.

 \bigskip

\noindent{\sl Proof of Theorem \ref{teo:CohomRing}. } Let $S$ be an (abstract) exterior ring defined by the presentation in the statement, i.e., generated by the elements $\tilde\sigma_v$ for $v\in V(\Gamma)$ and $\tilde\sigma_e$ for $e=\{u,v\}$ edge in $E(\Gamma)$ with label bigger than 2 subject to:
\begin{itemize}
\item[(1)] $\tilde\sigma_u\tilde\sigma_v=m\tilde\sigma_e$ if $e=\{u,v\}$ has label $2m$ and $u<v$,
\item[(2)] $\tilde\sigma_{X_1}\tilde\sigma_{X_2}=0$ if $X_1\cap X_2\neq\emptyset$ or $X_1\cup X_2$ is not spherical. 
\end{itemize}
Items i) and ii) of Lemma \ref{CupprodLemma} imply that $\tilde\sigma_X\mapsto\sigma_X$ induces a well-defined (graded) ring homomorphism $\varphi:S\to H^\bullet(A_\Gamma,R)$. Moreover, $H^\bullet(A_\Gamma,R)$ has a free $R$-basis of the form $\{\sigma_X\mid X\subseteq\Gamma\text{ spherical}\}$ and as $\Gamma$ is even, all the spherical subgroups decompose as direct products of spherical subgroups associated to either single vertices of edges. By item iii) of Lemma \ref{CupprodLemma} this implies that $\varphi$ is an epimorphism.

To see that $\varphi$ is indeed an isomorphism we claim that for $X_1,\dots,X_n\subseteq\Gamma$ such that each $X_i$ is either a single vertex or an edge we have 
	$\tilde\sigma_{X_1}\cdots\tilde\sigma_{X_n}=0$ in $S$ if:
	\begin{itemize}
		\item $X_i\cap X_j\neq\emptyset$ for some $i,j\in\{1,\dots,n\}$, or
        \item they are pairwise disjoint and $X:=X_1\cup\cdots\cup X_n$ is not spherical.
	\end{itemize}
Observe that this claim implies that for each $k$ the $R$-rank of the $k$-th degree term $S_k$ of $S$ is at most the $R$-rank of $H^\bullet(A_\Gamma,R)$. This together with the fact that $\varphi$ is epi yields the result.

	We treat the two alternatives separately. First, assume that $X_i\cap X_j\neq\emptyset$ for some $i\neq j$.
	Then, by a relator of type (2), we have
	$\tilde\sigma_{X_i}\tilde\sigma_{X_j}=0$, and thus the whole product vanishes. So we may assume that the $X_i$'s are pairwise disjoint and $X$ is not spherical. Since $\Gamma$ is even, one of the following holds:
	\begin{enumerate}
		\item[i)] $X$ is not complete or;
		\item[ii)] there exists a triangle $\Theta\subset X$ having at least two edges
		with labels greater than $2$.
	\end{enumerate}
	
	In case i), there exist indices $i\neq j$ such that
	$X_i\cup X_j$ is not complete. By a relator of type (2),
	$\tilde\sigma_{X_i}\tilde\sigma_{X_j}=0$, and therefore the whole product vanishes.
	
	In case ii), let $\Theta=\{u,v,w\}$ be such a triangle.
	Either there exist indices $i,j$ such that
	$X_i\cup X_j=\Theta$, in which case
	$\tilde\sigma_{X_i}\tilde\sigma_{X_j}=0$ by a relator of type (2), and hence the
	whole product vanishes; or the vertices of $\Theta$ appear as singletons among
	the $X_k$'s, say $X_i=\{u\}$, $X_j=\{v\}$, and $X_k=\{w\}$.
	Write $2m$ for the (even) label on the edge $\{u,v\}$. Using a relator
	of type (1), we obtain
	$$
	\tilde\sigma_u\tilde\sigma_{v}\tilde\sigma_{w}
	= \pm m\,\tilde\sigma_{\{u,v\}}\tilde\sigma_w.
	$$
	Since $\{u,v\}\cup\{w\}$ is non-spherical by hypothesis,
	$\tilde\sigma_{\{u,v\}}\tilde\sigma_w=0$ by a relator of type (2), and hence the triple product
	(and therefore the entire product) vanishes.\qed
\subsection{The isomorphism problem of even Artin groups}\label{sec:Isomorphism}
An important open problem in the theory of Artin groups is to determine whether a group isomorphism $A_\Gamma \cong A_\Delta$ implies a (labeled) graph isomorphism $\Gamma \cong \Delta$. This is known to hold for arbitrary RAAGs (see \cite{Droms}), and it is conjectured to be true for arbitrary even Artin groups (see \cite{Blasco-García 2019}). In that work, Blasco-García and París proved the statement for Artin groups whose labels all lie in
$$\{2c,\infty\} \cup \{2d^r \mid r \geq 1\},$$
for some integers $c,d \geq 2$ with $\gcd(c,d) = 1$.

It is natural to expect that the cohomology ring of even Artin groups could be used to distinguish them; that is, one might hope that
$$A_\Gamma \cong A_\Delta \quad \text{if and only if} \quad H^\bullet(A_\Gamma,\mathbb{Z}) \cong H^\bullet(A_\Delta,\mathbb{Z}).$$
However, in this section we show that this is not the case, by constructing an example for which this equivalence fails.

Consider the Artin groups $A_\Gamma$ and $A_\Delta$ defined by the following graphs (from left to right):
		$$
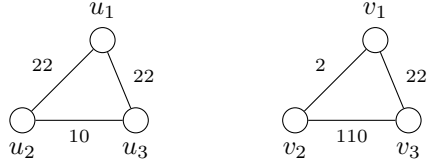
\begin{tikzpicture}[main/.style = {draw, circle},node distance={15mm}] 
		\node[label=below:{$u_2$}][main] (1)  {}; 
		\node[label={$u_1$}][main] (6) [above right of=1] {};
		\node[label=below:{$u_3$}][main] (2) [right of=1] {}; 
		
		\draw[-] (1) -- node[below] {$\scriptstyle{10}$}  (2);
		\draw[-] (1) -- node[above left]{$\scriptstyle{22}$}  (6);
		\draw[-] (6) -- node[pos=0.7,above right] {$\scriptstyle{22}$}   (2);
	\end{tikzpicture} \;\;\;\;\;\;\;\;\;\;\;\;\;\;\;\begin{tikzpicture}[main/.style = {draw, circle},node distance={15mm}] 
		\node[label=below:{$v_2$}][main] (1)  {}; 
		\node[label={$v_1$}][main] (6) [above right of=1] {};
		\node[label=below:{$v_3$}][main] (2) [right of=1] {}; 
		
		\draw[-] (1) -- node[below] {$\scriptstyle{110}$}  (2);
		\draw[-] (1) -- node[above left]{$\scriptstyle{2}$}  (6);
		\draw[-] (6) -- node[pos=0.7,above right] {$\scriptstyle{22}$}   (2);
	\end{tikzpicture}  $$
\begin{proposition}
$H^\bullet(A_\Gamma,\mathbb{Z}) \cong H^\bullet(A_\Delta,\mathbb{Z})$.
\end{proposition}
\begin{proof}
   Consider the map $f: H^\bullet(A_\Gamma,\mathbb{Z}) \cong H^\bullet(A_\Delta,\mathbb{Z})$ defined by
\begin{align*}
f(\sigma_{u_1}) &= 11\sigma_{v_1} + 2\sigma_{v_3}, \\
f(\sigma_{u_2}) &= \sigma_{v_2} + \sigma_{v_3}, \\
f(\sigma_{u_3}) &= 5\sigma_{v_1} + \sigma_{v_3}, \\
f(\sigma_{u_1u_2}) &= \sigma_{v_1v_2} + \sigma_{v_1v_3} - 10\sigma_{v_2v_3}, \\
f(\sigma_{u_1u_3}) &= \sigma_{v_1v_3}, \\
f(\sigma_{u_2u_3}) &= -\sigma_{v_1v_2} - 11\sigma_{v_1v_3} + 11\sigma_{v_2v_3}.
\end{align*}
It is easily checked that this map induces isomorphisms $H^i(A_\Gamma,\mathbb{Z}) \cong H^i(A_\Delta,\mathbb{Z})$ for $i=1,2$, and that $f(\sigma_{u_i}\sigma_{u_j}) = f(\sigma_{u_i})\,f(\sigma_{u_j})$
for all $i,j \in \{1,2,3\}$. Hence $f$ respects cup products, and therefore it is an isomorphism of graded rings.   
\end{proof}
To prove that $A_\Gamma$ and $A_\Delta$ are not isomorphic, we will use the $\Sigma^1$-invariant. This is a group invariant of geometric nature that forms part of the bigger family of $\Sigma$-invariants, also called BNRS-invariant after Bieri, Neumann, Renz and Strebel (cf. \cite{BieriRenz 1988, BieriNeumannStrebel 1987}).

\begin{definition}
Let $G$ be a finitely generated group and $S$ a generating set of $G$. The \textit{character sphere} of $G$ is defined as
$$S(G)=\left(\mathrm{Hom}(G,\mathbb{R})\setminus\{0\}\right)/\sim,$$
where $\chi_1\sim\chi_2$ if and only if there exists $t>0$ such that $\chi_1=t\chi_2$. The \textit{$\Sigma^1$-invariant} of $G$ is defined as
$$\Sigma^1(G)=\{[\chi]\in S(G)\mid \mathrm{Cay}_\chi(G,S)\text{ is connected}\}.$$
Here, $\mathrm{Cay}(G,S)$ denotes the Cayley graph of $G$ with respect to the generating set $S$, and $\mathrm{Cay}_\chi(G,S)$ is the subgraph induced by the vertices $v$ with $\chi(v)\geq 0$.
\end{definition}

$\Sigma^1$-invariants of triangular Artin groups were computed in \cite[Lemma 7.13]{Escartín-MP}.

\begin{lemma}
If $A_\Omega$ is a triangular Artin group, then $[\chi]\in\Sigma^1(A_\Omega)$ if and only if $\mathrm{Liv}^\chi$ is connected.
\end{lemma}

Here, $\mathrm{Liv}^\chi$ is the subgraph of $\Omega$ obtained by deleting all vertices $v\in V(\Omega)$ with $\chi(v)=0$ and all edges $e=\{u,w\}\in E(\Omega)$ with $\chi(u)+\chi(w)=0$ and even label $\geq 4$.

Using this result, we obtain a criterion to distinguish  triangular even Artin groups.

\begin{proposition}\label{TriangleArtin}
If two triangular even Artin groups are isomorphic, then their defining graphs have the same number of edges labeled $2$.
\end{proposition}

\begin{proof}
Consider the group $A_\Omega$, where $\Omega$ is the following graph:
$$
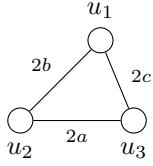
\begin{tikzpicture}[main/.style = {draw, circle},node distance={15mm}] 
    \node[label=below:{$u_2$}][main] (1)  {}; 
    \node[label={$u_1$}][main] (6) [above right of=1] {};
    \node[label=below:{$u_3$}][main] (2) [right of=1] {}; 
    
    \draw[-] (1) -- node[below] {$\scriptstyle{2a}$}  (2);
    \draw[-] (1) -- node[above left]{$\scriptstyle{2b}$}  (6);
    \draw[-] (6) -- node[pos=0.7,above right] {$\scriptstyle{2c}$}   (2);
\end{tikzpicture}$$
with $a,b,c\geq 2$. Then $S(A_\Omega)$ is the 2-sphere $\mathbb{S}^2$ and
$\mathrm{Liv}^\chi$ is disconnected if and only if one of the following holds:
\begin{itemize}
    \item $\chi(u_i)=0$ and $\chi(u_{i+1})+\chi(u_{i+2})=0$ for $i\in\{1,2,3\}$;
    \item $\chi(u_i)+\chi(u_{i+1})=0$ and $\chi(u_i)+\chi(u_{i+2})=0$ for $i\in\{1,2,3\}$,
\end{itemize}
where the indices are taken modulo $3$. In particular, $\Sigma^1(A_\Omega)$ is homotopy equivalent to $\mathbb{S}^2$ with $12$ points removed, so $\pi_1(\Sigma^1(A_\Omega))\cong F_{11}$. Similarly,  
\begin{enumerate}
    \item if $\Omega$ has precisely one edge labeled $2$, $\Sigma^1(A_\Omega)$ is $\mathbb{S}^2$ with $6$ points removed so $\pi_1(\Sigma^1(A_\Omega))\cong F_5$;

    \item  if $\Omega$ has precisely two edges labeled $2$, then $\Sigma^1(A_\Omega)$ is $\mathbb{S}^2$ with $2$ points removed so $\pi_1(\Sigma^1(A_\Omega))$ is cyclic and
    
   \item if all three edges of $\Omega$ are labeled $2$, then the group $A_\Omega$ is free abelian and  $\Sigma^1(A_\Omega)=S(A_\Omega)$.
\end{enumerate}
Therefore, the claim follows.
\end{proof}
\begin{corollary} The groups
    $A_\Gamma$  and $A_\Delta$ are not isomorphic.
\end{corollary}
However, even though the cohomology ring does not completely determine even Artin groups, it does give an obvious invariant: the number of spherical subgraphs of a given degree.

\begin{proposition}
Let $A_\Gamma$ and $A_\Delta$ be two even Artin groups with $A_\Gamma \cong A_\Delta$. Then the number of spherical subgraphs on $n$ vertices is the same in both $\Gamma$ and $\Delta$. In particular, $\lvert E(\Gamma)\rvert = \lvert E(\Delta)\rvert$.
\end{proposition}

\begin{proof}
If $A_\Gamma \cong A_\Delta$, then $H^\bullet(A_\Gamma,\mathbb{Z}) \cong H^\bullet(A_\Delta,\mathbb{Z})$, and $\dim_{\mathbb{Z}} H^n(A_\Gamma,\mathbb{Z}) = \dim_{\mathbb{Z}} H^n(A_\Delta,\mathbb{Z})$ for all $n \geq 1$, so the result follows.
\end{proof}
\section{\texorpdfstring{$p$}--Restrcited Zassenhauss Lie algebras of Artin groups}\label{sec:central}


Given a group $G$, there are several algebraic structures that can be defined in terms of $G$ and give important information about the group. One is, for example, the cohomology ring that we have considered above. But probably the most basic such structure is the group ring $RG$ where $R$ is a ring, typically commutative and unital. In this paper we are going to see group rings as filtered rings, using the natural descending filtration by powers of the augmentation ideal $\omega_R$ which is the kernel of the map $RG\to R$ that sends all the group elements to 1:
\begin{equation}\label{eq:filtRG}
    RG=\omega_R^0\geq\omega_R\geq\omega_R^2\geq\ldots\geq\omega_R^i\geq\ldots
\end{equation}


This filtration is closely related to a filtration of the group $G$ itself. The descending central series $\gamma_i(G)$ is defined inductively by $\gamma_1(G)=G$ and $\gamma_i(G)=[G,\gamma_{i-1}(G)]$ and yields the {\sl Magnus} Lie ring
$$\gr_\bullet(G)=\bigoplus_{i\geq 1}\gamma_i(G)/\gamma_{i+1}(G).$$
The terms of $\gr_\bullet(G)$ can have torsion, but it is possible to define a similar torsion-free Lie ring using the {\sl isolator} central series instead, defined by
$$\gamma_i^{[0]}(G)=\{x\in G\mid x^m\in\gamma_i(G)\text{ for some }m>0\}.$$
It is well-known that the quotients $\gamma^{[0]}_i(G)/\gamma^{[0]}_{i+1}(G)$ are torsion-free, and from this one deduces that 
$$\gr^{[0]}_\bullet(G):=\bigoplus_{i\geq 1}(\gamma^{[0]}_i(G)/\gamma^{[0]}_{i+1}(G))\otimes\Q=\gr_\bullet(G)\otimes_\Z\Q.$$

For each prime $p$ there are also a $p$-local versions, like the series 
$\lambda_i(G)=[\lambda_{i-1}(G),G]\lambda_{i-1}(G)^p$. In this paper we are going to consider
the series
$$\gamma_i^{[p]}(G)=\prod_{jp^s\geq i}\gamma_j(G)^{p^s}$$
that was first introduced by Zassenhaus in 1939.

The relationship between these central series and the filtration (\ref{eq:filtRG}) is the following. The {\sl dimension subgroups} of $G$ with coefficients in $R$ are the groups
$$D_i^R(G)=\{g\in G\mid 1-g\in w_R^i\}.$$
These groups form a central series and in particular, $D_1^R(G)=G$. In the case when $R=\Z$ we have $\gamma_i(G)\leq D^\Z_i(G)$ but in general this inequality can be strict \cite{Rips} and the dimension subgroups $D^\Z_i(G)$ are rather mysterious. However, as Sjogren showed in \cite{Sjogren}, the quotient $D^\Z_i(G)/\gamma_i(G)$ is an abelian group of bounded torsion. Bartholdi and Mikhailov have recently constructed examples of groups where the quotient $D^\Z_i(G)/\gamma_i(G)$ contains any prescribed abelian group of bounded exponent.
However, the situation changes if we consider coefficients in a field:

\begin{theorem}[Quillen-Jennings-Lazard] Let $p$ be zero or a prime and $K=\Q$ if $p=0$, $K=\F_p$ in other case. Then
$$\gamma_i^{[p]}(G)=D_i^{K}(G),$$
\end{theorem}

 In the case of prime characteristic, this result was first shown by Jennings for finite $p$ groups  \cite{JenningsP} and by Lazard for arbitrary groups  \cite{Lazard}. 
Jennings also claimed the case of characteristic zero  \cite{Jennings0}, but he never published a proof. Later, Hall gave a proof that he attributes to Jennings and that can be read in the notes by Pengitore \cite{Hall}. Both cases where also shown independently  and with completely different methods by Quillen in \cite{Quillen}.

As we did for $\gamma_i(G)$ and  $\gamma_i^{[0]}(G)$, using $\gamma_i^{[p]}(G)$ we may define a Lie algebra
$$\gr^{[p]}_\bullet(G)=\bigoplus_{i\geq 1}\gamma^{[p]}_i(G)/\gamma^{[p]}_{i+1}(G).$$
This algebra has extra structure: it is $p$-restricted. Recall that a $p$-{\sl restricted Lie algebra} is a Lie algebra $L$ over a field $F$ of characteristic $p$ with an extra operation $\mathfrak{a}\mapsto\mathfrak{a}^{[p]}$ such that for any $\mathfrak{a},\mathfrak{b}\in L$, $t\in F$:
\begin{itemize}
\item[i)] $[\mathfrak{a}^{[p]},\mathfrak{b}]=[\mathfrak{a},\buildrel{p}\over\ldots,\mathfrak{a},\mathfrak{b}]$,

\item[ii)] $(t\mathfrak{a})^{[p]}=t^p\mathfrak{a}^{[p]}$,

\item[iii)] $(\mathfrak{a}+\mathfrak{b})^{[p]}=\mathfrak{a}^{[p]}+\mathfrak{b}^{[p]}+\sum\limits_{i=1}^{p-1}s_i(\mathfrak{a},\mathfrak{b})$ where $s_i(\mathfrak{a},\mathfrak{b})$ is $\frac{1}{i}$ times the coefficient of $t^{i-1}$ in $[t\mathfrak{a}+\mathfrak{b},\buildrel{p-1}\over\ldots,t\mathfrak{a}+\mathfrak{b},\mathfrak{a}]$.
\end{itemize}

If $A_\Gamma$ is a RAAG, the algebras $\gr_\bullet(A_\Gamma)$ and $\gr^{[p]}_\bullet(A_\Gamma)$ are precisely the objects given by the analogous presentation as the group but in the corresponding category (in the case of the Magnus Lie ring this was first shown by Duchamp and Krob, see \cite{DuchampKrob}, \cite{BHS}). To avoid confusions, we will use fraktur letters to denote the vertices of $V(\Gamma)$ when we are denoting generators of a Lie algebra. So we have
$$\gr_\bullet(A_\Gamma)=\langle \mathfrak{a}\in V(\Gamma)\mid [\mathfrak{a},\mathfrak{b}]=0\text{ if }\{\mathfrak{a},\mathfrak{b}\}\text{ form an edge of }\Gamma\rangle$$
and $\gr^{[p]}_\bullet(A_\Gamma)$ has the same presentation but in the category of $p$-restricted Lie algebras. In the proof of these results, the fact that the terms of the series $\gamma_i(A_\Gamma)/\gamma_{i+1}(A_\Gamma)$ are all torsion-free plays an important role and as we will see later, this is no longer true for arbitrary Artin groups.

The simplest possible family of Artin groups which are not RAAGs are dihedral groups, i.e., Artin groups $G=A_e$ where the defining graph has a single edge $e$. If the edge has an odd label, then the abelianization of $G$ has rank one and in fact by Proposition 1 of \cite{BGG} we have
$$\gamma_i(G)=\gamma_2(G)$$
for any $i$, so the Magnus Lie ring is just $\Z$ concentrated at degree 1.
 If the label of the edge is an even number $2m$, it seems difficult to determine the Magnus Lie algebra $\gr_\bullet(G)$. The problem is related to the existence of torsion in $\gr_2(G)=\gamma_2(G)/\gamma_3(G)$: putting  $s=ab$, the relator $(ab)^m=(ba)^m$ can be rewritten as
$$as^ma^{-1}=s^m$$
and therefore
$$G=\langle a,s\mid (s^m)^a=s^m\rangle=BS(m,m),$$
and $1=[s^m,a]\equiv [s,a]^m$ modulo $\gamma_3(G)$. To see how  presence of torsion in $\gr_2(G)$ can make the Magnus Lie ring complicated,  see  \cite{Labute1} for a presentation of the Lie ring  $\gr_\bullet(\Z\ast\F_p)$ for $p$ a prime. In fact the algebra $\gr_\bullet(\Z\ast\F_p)$ is closely related to $\gr_\bullet(G)$: Bardakov and Neshadim show

\begin{theorem}\cite[Proposition 2]{BardakovNeshchadim}\label{teo:BN} Let $G=\langle a,b\mid (ab)^m=(ba)^m\rangle$ be an even dihedral Artin group. Then the terms of the descending Lie ring $\gr_\bullet(G)$ of $G$ are
$$\gr_1(G)=\Z^2,$$
$$\gr_i(G)=\gr_i(\Z\ast\Z_m)\text{ for }i\geq 2.$$
\end{theorem}

To avoid these kind of problems, one may consider the rationalized version
$\gr_\bullet(G)\otimes_\Z\Q$ but it turns out to be not very interesting. For example, for the dihedral Artin group above, it is easy to check that all the quotients $\gamma_{i-1}(G)/\gamma_{i}(G)$ are torsion for $i\geq 2$ (basically, because $\gamma_2(G)/\gamma_3(G)$ is) and therefore 
$\gr_\bullet(G)\otimes_\Z\Q$ is just $\Q$ concentrated in degree 1. In fact, it follows from results in \cite{KM} that for any Artin group $A_\Gamma$, $\gr_\bullet(A_\Gamma)\otimes_\Z\Q$ is the $\Q$-Lie algebra generated by $V(\Gamma)$ subject to the relations
$$\begin{aligned}
 [\mathfrak{a},\mathfrak{b}]=0\text{ if }\{\mathfrak{a},\mathfrak{b}\}
\text{ is an edge of $\Gamma$ of even label},\\
 \mathfrak{a}=\mathfrak{b}\text{ if }\{\mathfrak{a},\mathfrak{b}\}
\text{ is an edge of $\Gamma$ of odd label}.\\
\end{aligned}$$ 
This means that $\gr^{[0]}_\bullet(A_\Gamma)$ is precisely the rational Magnus Lie algebra of a quotient of $A_\Gamma$: the RAAG with defining graph the graph that one obtains from $\Gamma$ by identifying all the endpoints of edges with odd label, removing double edges, and labeling all the remaining edges with a 2. So we see that a lot of the information of the group is killed when we tensor with $\Q$.

The $p$-Zassenhaus series however can be seen as a local object, and therefore one can expect it to keep $p$-local info. This is precisely the case, as we will prove next. We will consider first the case of dihedral Artin groups and then we will extend to other Artin groups using the following result that generalizes a Theorem of Lichtmann for free products.

\begin{theorem}\cite{LMPW}\label{teo:Giorgio}  Let $p$ be either zero or a prime and $G_1,G_2$ be groups with a common subgroup $H\leq G_1,G_2$ which is strictly $p$-embedded in both.
Then 
$$\gr_\bullet^{[p]}(G_1\ast_HG_2)=\gr_\bullet^{[p]}(G_1)\ast_{\gr_\bullet^{[p]}(H)}\gr_\bullet^{[p]}(G_2).$$
\end{theorem}

Here, we say that a subgroup $H\leq G$ is {\sl strictly $p$-embedded} in $G$ is for any $i\geq 1$,
$$H\cap\gamma_i^{[p]}(G)=\gamma_i^{[p]}(H).$$
Retracts are strictly $p$-embedded for any $p$ so if $A_\Gamma$ is an even Artin group and $\Delta\subseteq\Gamma$, the group $A_\Delta$ is strictly $p$-embedded in $A_\Gamma$.
Unfortunately, this is no longer true if we drop the hypothesis that $A_\Gamma$ is even so our method can not be applied for arbitrary Artin groups.

Let $A_\Gamma$ be an Artin group. We denote by $\Gamma_p$ the graph obtained from $\Gamma$ as follows:

\begin{itemize}
    \item[(1)] Identify endpoints of edges with an odd label and remove the loops that one creates,
    
\item[(2)] Change all labels of the form $2kp^t$ with $k$ coprime to $p$ into $2p^t$.

\item[(3)] Identify all possible double edges to an edge with the smallest possible label between those being identified.
\end{itemize}



The graph $\Gamma_p$ can be seen as the even $p$-part of $\Gamma$ and it carries the basic $p$-local info of $A_\Gamma$. We are about to prove Theorem \ref{teo:prop}, which makes this explicit because it means that the graph $\Gamma_p$ gives precisely the pro-$p$ completion $\hat{G}=\varprojlim_{U\leq_pG}G$ of the Artin group $G=A_\Gamma$. 

 \bigskip

\noindent{\sl Proof of Theorem \ref{teo:prop} } Let $T$ the normal subgroup of $A_\Gamma$ generated by the elements of the form $ab^{-1}$ whenever $\{a,b\}\in E(\Gamma)$ has an odd label and $(ab)^{p^t}(ba)^{-p^t}$ whenever $\{a,b\}\in E(\Gamma)$ has label $2kp^t$ for $k$ coprime to $p$. Then 
$$A_{\Gamma_p}=A_\Gamma/T.$$
We claim that $T\leq \bigcap_{i\geq 1}\gamma_i^{[p]}(G)$. This will imply that $A_\Gamma$ and $A_{\Gamma_p}$ have the same pro-$p$ completion so we will get the result.

If $e=\{a,b\}\in E(\Gamma)$ has odd label, then $a$ and $b$ have the same image in the abelianization of $A_e$ so $ab^{-1}\in\gamma_2(A_e)$.
 By Proposition 1 of \cite{BGG}, $\gamma_r(A_e)=\gamma_2(A_e)$ for any $r\geq 2$. Therefore 
$$ab^{-1}\in\gamma_r(A_e)\leq\gamma_r(A_\Gamma)$$  
so for any $i$, choosing $r\geq i$ we have
$$ab^{-1}\in\gamma_r(A_\Gamma)\leq\prod_{jp^s\geq i}\gamma_j(A_\Gamma)^{p^s}=\gamma_i^{[p]}(A_\Gamma),$$  
i.e.
$$ab^{-1}\in  \bigcap_{i\geq 1}\gamma_i^{[p]}(G).$$  

Now, assume that $\{a,b\}\in E(\Gamma)$ has label $2kp^t$ for $k$ coprime to $p$. Let $\overline{G}=G/\gamma^{[p]}_i(G)$ which is a finite $p$-group of order, say, $p^r$, where $l$ depends on $i$. Choose integers $s,l$ such that $ks+lp^r=1$. Then in $\overline{G}$ we have
$$(\overline{a}\overline{b})^{p^t}=((\overline{a}\overline{b})^{ks})^{p^t}=(\overline{b}\overline{a})^{kp^ts}=((\overline{b}\overline{a})^{ks})^{p^t}=(\overline{b}\overline{a})^{p^t}$$
thus $(ab)^{p^t}(ba)^{-p^t}\in\gamma_i^{[p]}(G)$ for any $i$ thus $(ab)^{p^t}(ba)^{-p^t}\in  \bigcap_{i\geq 1}\gamma_i^{[p]}(G)$.
\qed

\begin{definition}
    We say that an Artin group $A_\Gamma$ is pro-$p$ FC if its  pro-$p$ completion is of type FC, meaning that the group $A_{\Gamma_p}$ is of type FC.
\end{definition}

As in the case of Magnus Lie algebras, $p$-Zassenhaus Lie algebras of dihedral Artin groups where the edge $e$ has odd label is not very interesting (see the proof of Theorem \ref{teo:prestricted} below) so we consider only the even case for now.

\begin{proposition}\label{prop:pZdihedral} Let $G=\langle a,b\mid (ab)^m=(ba)^m\rangle$ be an even dihedral Artin group. Put $m=p^tk$ with $k$ coprime with $p$. Then
$$\gr^{[p]}_\bullet(G)=\langle \mathfrak{a},\mathfrak{b}\mid[(\mathfrak{a}+\mathfrak{b})^{[p]^t},\mathfrak{a}]=0\rangle.$$
\end{proposition}
\begin{proof} The group $G$ is the even dihedral Artin group associated to a graph with a single edge of label $2p^tk$.
If we denote by $G_p$ the even dihedral Artin group of an edge with label $2p^t$, then $G_p$ is the quotient of $G$ that we have considered in the proof of Theorem \ref{teo:prop}. In particular this means that $G$ and $G_p$ have the same $p$-restricted Zassenhaus Lie algebra, so we may assume $G=G_p$, i.e., we assume that $k=1$.

Set $s=ba$, and consider the presentation 
$$G=\langle a,s\mid [s^{p^t},a]=1\rangle.$$
The group $Z=\langle s^{p^t}\rangle$ is central and the quotient $Q=G/Z$ admits the following presentation
$$\langle a,s\mid s^{p^t}=1\rangle,$$
i.e., $Q\cong\Z\ast\Z_{p^t}$. As the abelianization of $G$ is the free abelian group generated by the cosets of $a$ and $s$, we have non trivial elements
$\bar{a}=a\gamma_2^{[p]}(G)$, $\bar{s}=s\gamma_2^{[p]}(G)$ in $\gr_1^{[p]}(G)$ such that $[\bar{s}^{p^t},\bar{a}]=0$. 


Now, let $L_\bullet$ be the $p$-restricted Lie algebra given by the presentation
$$L_\bullet=\langle \mathfrak{a},\mathfrak{s}\mid[\mathfrak{s}^{[p]^t},\mathfrak{a}]=0\rangle.$$

The previous remarks imply that the map $\mathfrak{a}\mapsto\bar{a}$, $\mathfrak{s}\mapsto\bar{s}$ extends to  a well defined homomorphism of $p$-restricted Lie algebras
$$\rho:L_\bullet\to\gr^{[p]}_\bullet(G).$$
Moreover, as $\bar{a}$ and $\bar{s}$ generate $\gr^{[p]}_\bullet(G)$ (because they generate the abelianization $G/G'$), $\rho$ is an epimorphism. We claim that $\rho$ is an isomorphism, this will imply the result because if we set $\mathfrak{b}=\mathfrak{s}-\mathfrak{a}$ in $L_\bullet$, then $L_\bullet$ has a presentation as in the statement.


Recall that $Z=\langle s^{p^t}\rangle$ and consider the short exact sequence
$$1\to Z\to G\to Q\to 1.$$
This induces a short exact sequence of $p$-restricted  Lie algebras
$$0\to K_\bullet\to \gr^{[p]}_\bullet(G) \to \gr^{[p]}_\bullet(Q)\to 0$$
where $K_\bullet$ has as degree $i$ term
$$K_i=Z\cap \gamma_i^{[p]}(G)/Z\cap \gamma_{i+1}^{[p]}(G).$$
The group $\langle s\rangle$ is a retract of $G$: to see it consider the morphism $G\to\langle s\rangle$
induced by $a\mapsto 1$, $s\mapsto s$. So it is strictly $p$-embedded thus for any $i\geq 1$ we have 
$$\langle s\rangle\cap\gamma_i^{[p]}(G)=\gamma_i^{[p]}(\langle s\rangle)=\langle s^{p^{t(i)}}\rangle$$
where $t(i)=\lceil \log_p(i)\rceil$ is the integer part by excess of the logarithm of $i$ in base $p$, equivalently, the smallest integer such that $p^{t(i)}\geq i$. 
Therefore
$$Z\cap \gamma_i^{[p]}(G)=Z\cap \langle s\rangle\cap\gamma_i^{[p]}(G)=Z\cap\langle s^{p^{t(i)}}\rangle=\langle s^{p^{\mathrm{max}\{t,t(i)\}}}\rangle$$
and
$$Z\cap \gamma_i^{[p]}(G)=\Bigg\{
\begin{aligned}
&\langle s^{p^t}\rangle\text{ if }i\leq p^t\\
&\langle s^{p^{t(i)}}\rangle\text{ otherwise}\\
\end{aligned}$$
thus $K_i=0$ for $0\leq i<p^t$ and for bigger values, $K_i=\F_p$ precisely if $i$ is a power of $p$ and $K_i=0$ otherwise.

Now, let $I$ be the ideal of $L_\bullet$ generated by the element $\mathfrak{s}^{[p]^t}$. Using $\rho$ we have a commutative diagram

\begin{equation}\label{eq:CD}
\begin{tikzcd}
        I \arrow[r, hook]\arrow[d,"\rho_1"] &L_\bullet \arrow[r,two heads]\arrow[d,"\rho"] &L_\bullet/I\arrow[d,"\rho_2"] \\
       K_\bullet \arrow[r, hook]&\gr_\bullet(G) \arrow[r,two heads]&\gr_\bullet(Q)\\
    \end{tikzcd}
\end{equation}

Moreover, it follows from the presentation of $L_\bullet$ that $I$ is central in $L_\bullet$ and the quotient admits the presentation
$$L_\bullet/I=\langle \mathfrak{a}, \mathfrak{s}\mid \mathfrak{s}^{[p]^t}=0\rangle.$$ 
This is precisely the free product of the Zassenhauss $p$-Lie algebras of the infinite cyclic group and of the group $\Z_{p^t}$. 
Using Theorem \ref{teo:Giorgio} (or Lichtmann's Theorem)
$$L_\bullet/I\cong\gr_\bullet^{[p]}(\Z)\ast\gr_\bullet^{[p]}(\Z_{p^t})\cong\gr_\bullet^{[p]}(\Z\ast\Z_{p^t})=\gr_\bullet^{[p]}(Q)$$
where the first isomorphism is induced by $\rho$. In other words, the map $\rho_2$ in the commutative diagram (\ref{eq:CD}) is an ismorphism. So to show that $\rho$ is an ismorphism we only have to show that  $\rho_1$ is. In fact, it is enough to show that $\rho_1$ is an isomorphism degree-wise.

The ideal $I$ is abelian and generated by a single element of  degree $p^t$. Taking this into account one easily sees that $I_i=L_i\cap I$ can be described as follows
$$I_i=\Bigg\{
\begin{aligned}
&\F_p\text{ if }i=p^s,s\geq t\\
&0\text{ otherwise}.\\
\end{aligned}$$
Therefore the map $\rho_1$ is indeed an isomorphism, as we wanted to prove.
 \end{proof}

 Using this Proposition we can proceed now to prove Theorem \ref{teo:prestricted}.  
 
  \bigskip
 
 \noindent{\sl Proof of Theorem \ref{teo:prestricted}.} Let $L_\Gamma$ denote the $p$-restricted Lie algebra 
 in the statement of Theorem \ref{teo:prestricted}.
 We claim that $L_\Gamma\cong\gr_\bullet^{[p]}(A_\Gamma)$. The proof of Theorem \ref{teo:prop} implies that there is an epimorphism
 $$A_\Gamma\to A_{\Gamma_p}$$
with kernel $T$ so that $T\leq\bigcap_{i\geq 1}\gamma_i^{[p]}(G)$ thus
$$\gr_\bullet^{(p)}(A_\Gamma)=\gr_\bullet^{(p)}(A_{\Gamma_p}).$$
So we may assume $\Gamma=\Gamma_p$, and in particular, we may assume that $\Gamma$ is even of FC-type. If $\Gamma$ is complete, then by the FC-type condition it must be a direct product of groups which are either dihedral even (with label of the form $2p^t$) of free abelian and the result follows by Proposition \ref{prop:pZdihedral}. In other case, we may find proper subgraphs $\Gamma_1,\Gamma_2\subsetneq$ such that $\Gamma=\Gamma_1\cup\Gamma_2$ and then
$$A_{\Gamma_1}\ast_{A_\Delta}A_{\Gamma_2}$$
for $\Delta=\Gamma_1\cap\Gamma_2$. Moreover, the fact that $\Gamma$ is even implies that both $A_{\Gamma_1}$ and $A_{\Gamma_2}$ are retracts of $A_\Gamma$ so the result follows by induction on the number of vertices of $\Gamma$ using Theorem \ref{teo:Giorgio}.\qed

\section{Residual properties and the proof of Theorem \ref{teo:isoproblem}}\label{sec:res}
If $\mathcal{P}$ is a group property, a group $G$ is said to be \emph{residually $\mathcal{P}$} if for every $1\neq g\in G$ there exists a quotient $G/N$ having property $\mathcal{P}$ such that $g\notin N$. Equivalently, for every $1\neq g\in G$ there exists a homomorphism $\varphi:G\to Q$, where $Q$ has property $\mathcal{P}$, such that $\varphi(g)\neq 1$. In particular, a group is residually-$p$ if and only if
$$\bigcap_i \gamma_i^{[p]}(G)=1.$$

The residual properties of Artin groups are not completely understood. It is conceivable that all Artin groups are residually finite \cite{BlascoParis2, Jankiewicz, Meier} and virtually residually-$p$ \cite{JS}.

Right-angled Artin groups are residually torsion-free nilpotent. However, it is a consequence of the following discussion that RAAGs are the only Artin groups with this property.

In Proposition~2 of \cite{BGG}, Bellingeri, Gervais, and Guaschi show that a dihedral Artin group is residually nilpotent if and only if the label is of the form $2p^t$ for some prime $p$. In that case, the group is isomorphic to the Baumslag--Solitar group $\mathrm{BS}(p^t,p^t)$, which is residually-$p$ (but not residually torsion-free nilpotent if $t>0$) \cite{Moldavanskii}. This implies that if an Artin group $A_\Gamma$ is residually-$p$, then necessarily it must be even and all the labels have to be of the form $2p^t$ for some $t$. Using the fact that being residually-$p$ is preserved under free amalgamated products along retracts \cite{BS}, we obtain the converse in the FC-type case:

\begin{lemma}\label{lem:resp}
Let $A_\Gamma$ be an even Artin group of FC-type such that all the labels in $\Gamma$ are of the form $2p^t$ for some $t$. Then $A_\Gamma$ is residually-$p$.
\end{lemma}

\begin{proof}
In the dihedral case, this follows from \cite{Moldavanskii}. If $\Gamma$ is complete, then $A_\Gamma$ is a direct product of groups which are either infinite cyclic or even dihedral Artin groups with labels a power of $p$. By the dihedral case, each factor is residually-$p$, and hence so is $A_\Gamma$. 

For the general case, note that such an Artin group can be constructed by iterating free amalgamated products along retracts, and this construction preserves the property of being residually-$p$ \cite{BS}.
\end{proof}

As a consequence, observe that if a group $A_{\Gamma}$ is pro-$p$ FC, then
$$A_\Gamma \big/ \bigcap_i \gamma_i^{[p]}(A_\Gamma)=A_{\Gamma_p}.$$
Using this, Theorem~\ref{teo:isoproblem} follows easily.

\bigskip

\noindent{\sl Proof of Theorem \ref{teo:isoproblem}}
If $A_\Gamma \cong A_\Delta$, it follows that
$$A_\Gamma \big/ \bigcap_i \gamma_i^{[p]}(A_\Gamma)\cong A_\Delta \big/ \bigcap_i \gamma_i^{[p]}(A_\Delta),$$
and since both groups are pro-$p$ FC, this implies that $A_{\Gamma_p}\cong A_{\Delta_p}$. By \cite[Theorem 5.1]{Blasco-García 2019}, it follows that $\Gamma_p\cong \Delta_p$.\qed
\begin{remark}
The converse of this result is not true, even if we assume that both groups are pro-$p$ FC for all primes $p$. As a counterexample consider the Artin groups $A_\Gamma$ and $A_\Delta$ defined by the following graphs (from left to right):
		$$
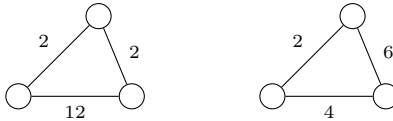
\begin{tikzpicture}[main/.style = {draw, circle},node distance={15mm}] 
		\node[][main] (1)  {}; 
		\node[][main] (6) [above right of=1] {};
		\node[][main] (2) [right of=1] {}; 
		
		\draw[-] (1) -- node[below] {$\scriptstyle{12}$}  (2);
		\draw[-] (1) -- node[above left]{$\scriptstyle{2}$}  (6);
		\draw[-] (6) -- node[pos=0.7,above right] {$\scriptstyle{2}$}   (2);
	\end{tikzpicture} \;\;\;\;\;\;\;\;\;\;\;\;\;\;\;\begin{tikzpicture}[main/.style = {draw, circle},node distance={15mm}] 
		\node[][main] (1)  {}; 
		\node[][main] (6) [above right of=1] {};
		\node[][main] (2) [right of=1] {}; 
		
		\draw[-] (1) -- node[below] {$\scriptstyle{4}$}  (2);
		\draw[-] (1) -- node[above left]{$\scriptstyle{2}$}  (6);
		\draw[-] (6) -- node[pos=0.7,above right] {$\scriptstyle{6}$}   (2);
	\end{tikzpicture}  $$
Then, the groups $A_{\Gamma_p}$ and $A_{\Delta_p}$ are isomorphic for all primes but $A_\Gamma$ is not isomorphic to $A_\Delta$ by Proposition \ref{TriangleArtin}.
\end{remark}
\section{\texorpdfstring{$p$}--Magnus maps}\label{sec:Magnus}

As we have seen at the beginning of Section \ref{sec:central}, for any group $G$ the group ring $RG$ has a descending filtration (\ref{eq:filtRG}) by powers of the augmentation ideal, called the {\sl augmentation filtration}. 

In general, we will say that an algebra $A$ is {\sl filtered} if there is a descending series of additive subgroups $F^iA$, $i\geq 0$ such that $(F^iA)(F^jA)\subseteq F^{i+j}A$.  We say that $A$ is {\sl separated} if $\bigcap_{i\geq 0} F^iA=0$. In that case there is an injection
$$A\to\varprojlim_i A/F^iA,$$
if this injection is an isomorphism we say that $A$ is {\sl complete}.

Associated to a filtered algebra $A$ we have a graded algebra
$$\gr_\bullet(A)=\bigoplus_{i\geq 0} F^iA/F^{i+1}A.$$

If $A,B$ are filtered algebras, a map $f:A\to B$ is {\sl filtered} if $f(F^iA)\subseteq F^iB$. 
Then $\gr(f):\gr_\bullet(A)\to\gr_\bullet(B)$ is the map given by
$$\gr(f)(a+F^iA)=f(a)+F^iB$$
where $a\in F^{i-1}A-F^{i}A$. It is a map of graded algebras.

We say that $f$ is an {\sl isomorphism of filtered algebras} if it is an isomorphism and $f(F^iA)=F^iB$. In that case, it is easy to see that $\gr(f)$ is also an isomorphism of graded algebras.  Conversely, if $A$ and $B$ are separated and complete and $\gr(f)$ is an isomorphism of graded algebras, then $f$ is an isomorphism of filtered algebras \cite[Lemma 2.4]{SuciuWangT}.


In the particular case of a group ring $RG$ with the augmentation filtration  we have
$$\gr_\bullet(RG)=\bigoplus_{i\geq 0}\omega_G^i/\omega_G^{i+1}.$$

Now, let $p$ be either zero or a prime and $F$ a field of characteristic $p$.  The (restricted) Lie algebra $\gr^{[p]}_\bullet(G)\otimes_\Z F$ has a universal envelopping algebra that we denote 
$$U_G=U(\gr^{[p]}_\bullet(G)\otimes_\Z F)$$
if the field $F$ is clear by the context.
This algebra is naturally graded and generated in degree 1. It is a remarkable Theorem of Quillen that these two objects are essentially the same:

\begin{theorem}[\cite{Quillen}]\label{Quillen} For a field $F$ of characteristic $p$ (zero or a prime) there is a natural isomorphism of graded algebras
$$U_G\cong\gr_\bullet(FG)$$
induced by $g\gamma_i^{[p]}(G)\mapsto g-1+\omega_G^i$.
\end{theorem}

We can complete $FG$ respect to the augmentation filtration and consider
$$F[[G]]=\varprojlim_i FG/\omega_G^i.$$
The augmentation filtration in $FG$ induces a filtration in $F[[G]]$ by powers of the completed augmentation ideal so $F[[G]]$ is a complete filtered algebra. If $F$ has prime characteristic $p$ and the group is residually-$p$, then $\bigcap_{i\geq 0}\omega_G^i=0$ so both $FG$ and $F[[G]]$ are separated.

In general, if $A=\bigoplus_{i\geq 0}A_i$ is a graded algebra, its {\sl degree completion} is
$$\hat{A}=\prod_{i\geq 0}A_i.$$
It is no longer graded, but filtered via
$$F^jA=\prod_{i\geq j}A_i$$
and one has $\gr_\bullet(\hat{A})=A$.
In particular, we will denote by $\hat{U}_G$ and $\hat{\gr}_\bullet(FG)$ the degree completions of  the graded algebras $U_G$ and ${\gr}_\bullet(FG)$.



\begin{definition}\label{def:filtered} With the previous notation, we say that a residually-$p$ group $G$ is {\sl $F$-Magnus} if there is an isomorphism of filtered complete algebras
$$M_G:F[[G]]\to\hat{U}_G.$$
If $F=\F_p$ we will say that $G$ is {\sl $p$-Magnus}. We will also say that $M_G$ is a {\sl Magnus map}.
\end{definition}

\begin{lemma} If $G$ is $F$-Magnus, then there is a Magnus map $M_G:F[[G]]\to\hat{U}_G$ such that
$$\gr(M_G):\gr(F[[G]])=\gr(F[G])\to\gr(\hat{U}_G)=U_G$$
    is the inverse of Quillen's map.
\end{lemma}
\begin{proof} As $G$ is $F$-Magnus, there is a Magnus map $M^1_G:F[[G]]\to\hat{U}_G$. Consider the completed Quilllen's map
$\hat{Q}_G:\hat{U}_G\to\hat{\gr}_\bullet FG$ and let $M_G$ be the composition
$$F[[G]]\buildrel{M^1_G}\over\to\hat{U}_G\buildrel{Q_G^{-1}}\over\to\hat{\gr}_\bullet FG\buildrel{\gr(M^1_G)^{-1}}\over\to\hat{U}_G.$$
As all the maps are isomorphisms of filtered complete algebras, so is $M_G$. The fact that $\gr(M_G)$ is inverse to $Q_G$ is obvious.
\end{proof}

\begin{remark}\label{rem:Giorgiofiltered} As a consequence, if $G$ is $F$-Magnus, then there is an isomorphism of filtered complete algebras
$$\phi:F[[G]]\to\hat{\gr}_\bullet FG$$
such that $\gr(\phi)$ is the identity map. Conversely, if there is an isomorphism of filtered complete algebras
$$F[[G]]\to\hat{\gr}_\bullet FG$$ 
then $G$ is $p$-Magnus and there is also such an isomorphism $\phi$ with $\hat{\gr}(\phi)$ the identity map. In fact this is precisely the definition given in \cite[Definition 5.5.1]{GiorgioThesis}, note that these groups are there called filtered formal. The reason for that terminology is the notion of $\Q$-filtered formality used in \cite[Proposition 2.5]{SuciuWangT}, observe however that according to \cite{SuciuWangT} a group is $\Q$-filtered formal if there is an isomorphism of complete filtered Hopf algebras $F[[G]]\to\hat{\gr}_\bullet FG$ and in our case we do not ask the map $\phi$ to preserve the Hopf algebra structure. 
\end{remark}


\begin{lemma}\label{lem:filteredembed} Let $G$ be a group with a strictly $p$-embedded subgroup $H\leq G$. If $G$ is $F$-Magnus, then so is $H$.
\end{lemma}
\begin{proof} By Remark \ref{rem:Giorgiofiltered} there is an isomorphism of filtered complete algebras $\phi_G:F[[G]]\to\hat{\gr}_\bullet FG$ such that $\hat{\gr}(\phi)$ is the identity. Moreover,
$H$ is strictly $p$-embedded in $H$, so there is a commutative diagram
\[\begin{tikzcd}
    F[[H]]\ar[d, hook,"\iota_1"]\ar[r,"\phi_H"]&\hat{\gr}_\bullet FH\ar[d, hook,"\iota_2"]\\
   F[[G]]\ar[r,"\phi_G"]&\hat{\gr}_\bullet FG\\
\end{tikzcd}
\]
where the vertical maps are inclusions.
The fact that $\hat{\gr}(\phi)$ is the identity implies that the restriction $\phi_H$ of $\phi_G$ to $F[[H]]$ is an isomorphism of filtered complete algebras $\phi_H:F[[H]]\to\hat{\gr}_\bullet FH$.
\end{proof}

We will need the following version of Proposition 5.1.2 of \cite{GiorgioThesis} for discrete groups

\begin{theorem}\label{teo:Giorgiofiltered} Let $G_1,G_2$ be groups with a common subgroup $H$ which is strictly $p$-embedded in both. If $G_1$ and $G_2$
 are $F$-Magnus, then so is $G=G_1\ast_HG_2$. \end{theorem}
 \begin{proof} The group $G=G_1\ast_HG_2$ is the pushout of the diagram
  \[\begin{tikzcd}
    &H\ar[d, hook]\ar[r,hook]&G_1\ar[d,dashed]\\
   &G_2\ar[r,dashed]&G\\
\end{tikzcd}
\] 
 The group ring functor has the functor that takes a ring to its units as a right adjoint, so it preserves push outs. Also the functor that takes a filtered ring to its completion has a right adjoint, the forgetful functor. Therefore $F[[G]]$ is the pushout of the induced diagram
   \[\begin{tikzcd}
    &F[[H]]\ar[d, hook]\ar[r,hook]&F[[G_1]]\ar[d,dashed]\\
   &F[[G_2]]\ar[r,dashed]&F[[G]]\\
\end{tikzcd}
\] 
 On the other hand, the hypothesis that $H$  is strictly $p$-embedded in both $G_1$ and $G_2$ together with Theorem \ref{teo:Giorgio}  implies that $\gr_\bullet FG$ is also the pushout of the corresponding diagram with  $\gr_\bullet FH$, $\gr_\bullet FG_1$ and $\gr_\bullet FG_2$ and again we deduce that the same holds true for the completed $\hat{\gr}_\bullet FG$. 
 
 As $G_1$ and $G_2$ are $F$-Magnus and $H$ is strictly $p$-embedded in both, arguing as in Lemma \ref{lem:filteredembed} we may choose $\phi_H,\phi_{G_1}$ and $\phi_{G_2}$ isomorphisms of filtered complete algebras such that the the images under $\hat{\gr}$ are all the identity and the following diagram commutes 
 \[\begin{tikzcd}
 \hat{\gr}_\bullet FH\ar[rrr,hook]\ar[ddd,hook]&&&\hat{\gr}_\bullet FG_1\ar[ddd]\\
    &F[[H]]\ar[d, hook]\ar[r,hook]\ar[ul,"\phi_H"]&F[[G_1]]\ar[d]\ar[ur,"\phi_{G_1}"]&\\
   &F[[G_2]]\ar[r]\ar[dl,"\phi_{G_1}"]&F[[G]]\ar[rd,dashed,"\phi_G"]&\\
   \hat{\gr}_\bullet FG_2\ar[rrr]&&&\hat{\gr}_\bullet FG\\
\end{tikzcd}
\] 
By the universal property of pushouts we deduce that there is a map $\phi_G$ of filtered complete separated algebras such that $\gr(\phi_G)$ is the identity. Therefore $\phi_G$ is an isomorphism  (see \cite[Lemma 2.4]{SuciuWangT}).
 \end{proof}

The existence of a $p$-Magnus map for a given group has strong cohomological consequences as we shall see next. 
Let $G$ be an arbitrary group and $\hat{G}$ its pro-$p$ completion. Let
$$dH^s(G,\F_p)=H^s_{cont}(G,\F_p)=\varinjlim_{H\trianglelefteq_pG}H^s(G/H,\F_p)$$
be the continuous cohomology of the pro-$p$ group $\hat{G}$. For the graded algebra $U_G$ we have bigraded Ext groups $\Ext_{U_G}^{s,i}(\F_p,\F_p)$ where $s$ is the cohomological degree and $i$ the internal degree. The algebra $\F_p[[G]]$ is filtered complete and $U_G=\gr(\F_p[[G]])$, so there is a May spectral sequence (\cite{May} for finite groups, \cite[Theorem 5.1.12]{SymondsWeigel} for pro-$p$ groups):
$$E_1^{j,i}=\Ext_{U_G}^{j+i,i}(\F_p,\F_p)\Rightarrow dH^{j+i}(G,\F_p),$$
where $i$ is the internal degree and $s=i+j$ the cohomological degree ($j$ is termed the complementary degree). We have:

\begin{lemma}\label{lem:collapse}
    If $G$ is $p$-Magnus, the spectral sequence collapses at page 1 so
    $$dH^{s}(G,\F_p)=\bigoplus_{i\geq 0}\Ext_{U_G}^{s,i}(\F_p,\F_p)$$
\end{lemma}
\begin{proof} If $G$ is $p$-Magnus, we have an isomorphism of filtered complete algebras 
$$\F_p[[G]]\cong\hat{U}_G$$
and as $\hat{U}_G$ is the degree completion of $U_G$, the spectral sequence collapses at page 1.
\end{proof}



\begin{proposition}\label{FMagnusArtin}
    Let $G=A_\Gamma$ be a pro-$p$ FC Artin group. Then $A_{\Gamma_p}$ is $F$-Magnus for any field $F$ of characteristic $p$.
\end{proposition}
\begin{proof}  Assume first that $\Gamma$ has just two vertices and one edge with label $2p^t$. Then
$$G=\langle a,b\mid (ab)^{p^t}=(ba)^{p^t}\rangle$$
and setting $s=ab$, $G$ also admits the presentation
$$G=\langle a,s\mid as^{p^t}=s^{p^t}a\rangle.$$
Let $U=U(\gr^{[p]}_\bullet(G)\otimes_{\F_p}F)$ be the universal enveloping algebra of the Zassenhaus restricted $p$-Lie algebra of $G$. Then Theorem \ref{teo:prestricted} implies that $U$ is precisely the $F$-algebra
$$U=F\langle a,s\mid as^{p^t}=s^{p^t}a\rangle.$$
Observe that there is an augmentation map $U\to F$, let $U^+$ be the augmentation ideal, i.e., the kernel of the augmentation map. This yields a filtration of $U$ by powers of $U^+$ but in fact, $U$ is naturally $\N$-graded. Let $\hat{U}$ be the degrre completion of $U$. We may define a representation of $G$ in the group of units of $\hat{U}$ as follows
$$\begin{aligned}
\varphi:G&\to\hat{U}^\times\\
a&\mapsto 1+a\\
s&\mapsto 1+s\\
\end{aligned}$$
This is well defined because we are in characteristic $p$ so in $U$ we have
$$(1+a)(1+s)^{p^t}=(1+a)(1+s^{p^t})=(1+s^{p^t})(1+a)=(1+s)^{p^t}(1+a)$$
and $\varphi$ can be extended to a map that we also denote  $\varphi:FG\to\hat{U}$. We obviously have $\varphi(\omega_G)\subseteq U^+$ so $\varphi$ is a map of filtered algebras that can be further extended to
$\hat{\varphi}:F[[G]]\to\hat{U}$. Note that the induced graded map $$\gr\hat{\varphi}:\gr_\bullet(F[[G]])=\gr_\bullet(F[G])\to\gr_\bullet\hat{U}=U$$
is precisely the inverse of  Quillen's map so it is an isomorphism and this implies that also $\hat{\varphi}$ is.

Now, iterating  Theorem \ref{teo:Giorgiofiltered} we get that for an arbitrary pro-p FC Artin group $A_\Gamma$ the group $A_{\Gamma_p}$ is p-magnus.
\end{proof}

\section{Cohomological \texorpdfstring{$p$}--completeness}\label{sec:complete}

In this section we will complete the proof of Theorem \ref{teo:filtered}.

\begin{definition}
    A group $G$ is called {\sl cohomologically $p$-complete} if the homomprphism $dH^s(G,\F_p)\to H^s(G,\F_p)$ induced by the natural map $G\to \hat{G}$ from $G$ to its $p$-completion is an isomorphism.
\end{definition}
\begin{proposition}\label{p-complete}
    Let $G=A_\Gamma$ be a pro-p FC Artin group. Then $A_{\Gamma_p}$ is cohomologically $p$-complete.
\end{proposition}
\begin{proof} Assume first that $\Gamma$ consists of just an edge with vertices $a,b$ and label $2p^t$. Then as we have seen before,
 $U$ is the graded $F$-algebra
$$U=F\langle a,s\mid as^{p^t}=s^{p^t}a\rangle.$$

There is a free $U$-resolution of the trivial module $F$

\begin{equation*}\begin{tikzcd}
0\arrow[r]&[0.1em]P_2\arrow[rr, "\delta_2"]\arrow[dr, two heads] &  &P_1\arrow[rr, "\delta_1"]\arrow[dr, two heads] & & P_0\arrow[r, "\delta_0"] &F\\
& &\ker(\delta_1)\arrow[ur, hook]&& \ker(\delta_0)\arrow[ur, hook]&\\
\end{tikzcd}\end{equation*}
where $P_2=U$, $P_1=U1_a\oplus U1_s$, $P_0=U$, $\delta_2(1)=1_as^2-1_ssa$, $\delta_1(1_a)=a$, $\delta_1(1_s)=s$ and $\delta_0$ is the augmentation map $a,s\mapsto 0$. To see that this is in fact a resolution, observe that each composition $\delta_i\delta_i$ vanish and moreover there is an $F$-split that we define next. The map $\delta_0$ splits via the $F$-injection $F\to P_0$. 
The $F$-algebra $U$ admits an $F$-basis consisting of elements of the form
\begin{equation}\label{eq:basis}
    s^ka^{l_1}s^{r_1}\ldots a^{l_m}s^{r_m}
\end{equation}
where $k\geq 0$ and either $m=0$ or $m>0$, $r_1,\ldots,r_m<p^t$ and $0<l_1,\ldots,l_m,r_1,\ldots,r_{m-1}$. The kernel $\ker(\delta_0)$ is the augmentation ideal $U^+$ and we let
 $\psi_1:\ker(\delta_0)\to P_1$ be the map given by
$$\psi_1(s^ka^{l_1}s^{r_1}\ldots a^{l_m}s^{r_m})=\Bigg\{\begin{aligned}
&s^{k-1}a^{l_1}s^{r_1}\ldots a^{l_m}s^{r_m}\text{ if }k>0\\
&a^{l_1-1}s^{r_1}\ldots a^{l_m}s^{r_m}\text{ otherwise}.
\end{aligned}$$
It is obvious that the composition $\delta_1\psi_1$ is the identity on $\ker(\delta_0)$ and easy to deduce that $\ker(\delta_1)=\im(\delta_2)$. Finally, observe that $\delta_2$ is injective because the image of the basis (\ref{eq:basis}) of $U$ is linearly independent.
Moreover, this is a graded free resolution and we see that $\delta_1$ has degree 1, whereas $\delta_2$ has degree 2. Using this resolution to compute the bigraded Ext functor $\Ext^{p,q}_U(F,F)$ where $p$ is the cohomological degree and $q$ the internal degree one sees that the only non-vanishing terms are
$\Ext_U^{0,0}(F,F)$,
$\Ext_U^{1,1}(F,F)$ and
$\Ext_U^{2,t+1}(F,F)$, with respective $F$-dimensions 1, 2 and 1. As cup products preserve both internal and cohomological degrees, we deduce that as a ring, $\Ext^{\bullet,\bullet}_U(F,F)$ is just the trivial ring of dimension 3 with a basis of three elements: $\gamma_a$, $\gamma_s$ of degree 1, $\gamma$ of degree $t+1$ such that all the products vanish.

As $A_\Gamma$ is $p$-Magnus, by Lemma \ref{lem:collapse} there is an isomorphism
$$\Ext_U^{\bullet,\bullet}(F,F)\cong dH^\bullet(G,F)$$
and a map
$$\tau:dH^\bullet(G,F)\to H^\bullet(G,F).$$
But by Theorem \ref{teo:CohomRing}, as $F$ is a field of characteristic $p$, $H^\bullet(G,F)$ is also the trivial ring of dimension 3, moreover the map $\tau$ sends basis to basis so it is an isomorphism and therefore $A_\Gamma$ is cohomologically $p$-complete.

Now, iterating Theorem Theorem C of \cite{LMPW} we get that $A_{\Gamma_p}$ is cohomologically $p$-complete for any pro-p FC Artin group $A_\Gamma$.
\end{proof}
Theorem \ref{teo:filtered} follows from Lemma \ref{lem:resp} and Propositions \ref{FMagnusArtin} and \ref{p-complete}.



\section{Groups based on Hydra Groups}\label{sec:Hydra}

An hydra group is a group with a presentation of the form
\begin{equation}\label{eq:hydra} H=\langle s,b\mid[s,\buildrel k\over\ldots,s,b]\rangle.\end{equation}
In \cite{BM}, Baumslag and Mikhailov show that hydra groups are residually torsion free nilpotent. 
In fact they give two proofs of this fact: the first, Theorem 1, is a direct proof for these groups and the second is a more general result where they consider one relator groups where the relator has the form
$$r=[v,v,\buildrel k\over\ldots,v,u]=1$$
so that $v$ and $u$ are words in two disjoint subsets of a minimal generating set. They show that for these kind of relators one can apply the following Theorem of Labute.
\begin{theorem}[\cite{Labute2}]
    Let $F$ be a free group with a free generating set $X$ and let $1\neq r\in \gamma_k(F)$ an element which is not a proper power modulo $\gamma_{k+1}(F)$. Let $G$ be the one relator group 
$$G=\langle X\mid r\rangle.$$
Consider the descending central series Lie rings $\gr_\bullet(F)$ and $\gr_\bullet(G)$ and let $\mathfrak{r}\in\gr_k(F)$ be the image of $r$. Then
$$\gr_\bullet(G)=\langle \mathfrak{X}\mid\mathfrak{r}\rangle$$
and $\gr_\bullet(G)$ is torsion free as abelian group.
\end{theorem}

Now, let $H$ be as in (\ref{eq:hydra}). We are in the situation of Labute's Theorem with $X=\{s,b\}$ and $r=[s,\buildrel k\over\ldots,s,b]$.
As Baumslag and Mikhailov argue, $r\in\gamma_{k+1}(F)$ and it is a basis element of the free abelian group $\gamma_{k+1}(F)/\gamma_{k+2}(F)$ so it is not a proper power and Labute's Theorem applies. In particular, 
$$\gr_\bullet(H)=\langle \mathfrak{s},\mathfrak{b}\mid[\mathfrak{s},\buildrel k\over\ldots,\mathfrak{s},\mathfrak{b}]\rangle$$
and $\gr_\bullet(H)$ is torsion free as abelian group. Therefore all the quotients $\gamma_i(H)/\gamma_{i+1}(H)$ are torsion free thus also are all the groups $H/\gamma_i(H)$.

Groups with this last property are called $\gamma$-free in \cite{MRW}. In that paper, the authors show that for these groups, for any prime $p$, the $p$-Zassenhauss Lie algebra $\gr^{[p]}(H)$ can be obtained applying the {\sl restrictification} functor to the Magnus Lie algebra $\gr_\bullet(H)$. Therefore
$$\gr^{[p]}_\bullet(H)=\langle \mathfrak{s},\mathfrak{b}\mid[\mathfrak{s},\buildrel k\over\ldots,\mathfrak{s},\mathfrak{b}]\rangle_{[p]}$$
where the subscript $[p]$ means that we consider this as a presentation in the category of $p$-restricted Lie algebras.

Using this fact we get:


\begin{theorem}\label{teo:Hydra}
    For any pro-$p$ FC Artin group $A_\Gamma$ and each prime $p$, there is a group $G$ which is residually torsion free nilpotent and such that
$$\gr_\bullet^{[p]}(A_\Gamma)=\gr_\bullet^{[p]}(G).$$
\end{theorem}
\begin{proof} Consider the graph $\Gamma_p$  and fix a total order in the set of vertices $V(\Gamma_p)$. Let $G$ be the group given by the following presentation. As generating set we have $V(\Gamma_p)$. As relators,
$$[ab,\buildrel{p^t}\over\ldots,ab,b]=1\text{ for each  $\{a,b\}\in E(\Gamma_p)$ with label $2p^t$ and $a<b$}.$$

Assume first that $\Gamma_p$ is a single edge with label $2p^t$ and vertices $a,b$. Put $s=ab$. Then $G$ is precisely the Hydra group of (\ref{eq:hydra}) and the result follows by the previous considerations. If $\Gamma_p$ is a complete graph,
the hypothesis that it is of type FC implies that $G$ is a direct product of groups which are either infinite cyclic or Hydra groups as (\ref{eq:hydra}) so the result follows again. Finally, if $\Gamma_p$ is not complete, then $G$ can be decomposed as a free amalgamated product of two groups of the same type but with a smaller number of vertices amalgamated along a retract. Using induction, we may assume that the result is true for the three components of the amalgamated product. And we deduce that $G$ is residually torsion free nilpotent from Theorem 6.2 in \cite{LMPW} and that $\gr_\bullet^{[p]}(G)=\gr_\bullet^{[p]}(A_\Gamma)$ from Theorem \ref{teo:prestricted}.
\end{proof}

\end{document}